\RequirePackage{fix-cm}
\documentclass{svjour3}                     
\smartqed  
\usepackage{graphicx}
\usepackage{amsmath,amssymb}
 \usepackage{mathptmx}      
\usepackage{latexsym}
 \journalname{} 
\begin{document}

\title{Elliptic boundary value problems  with  Gaussian white noise loads}

\thanks{This work has been funded 
by Academy of Finland (application number 250215, Finnish Programme for Centre of Excellence in Research 2012-2017) and European Research Council (ERC advanced grant 267700 - Inverse problems). }



\author{Sari Lasanen         \and
        Lassi Roininen \and Janne M.J. Huttunen 
}


\institute{S. Lasanen \at
              Mathematical Sciences,
              P.O.Box 3000,
           FI-90014   University of Oulu,
           Finland    
            \email{sari.lasanen@oulu.fi}           
\and 
Lassi Roininen\at 
 Tallinn University of Technology,
  Department of Mathematics,
  Ehitajate tee 5,
19086 Tallinn,
   Estonia
           \and
           Janne M.J. Huttunen,
               University of Eastern Finland, Department of Applied Physics,
                Yliopistonranta 1 F,
                 FI-70211 Kuopio,
                 Finland
                }

\date{} 

\maketitle

\begin{abstract}
Linear second order  elliptic boundary value problems (BVP)
on   bounded Lipschitz domains are studied in the case of   Gaussian white noise loads.
Especially, Neumann and Robin BVPs are considered.   

The main obstacle for applying the usual variational approach   is that the  Gaussian white noise has  irregular realizations. In particular, the corresponding  Neumann boundary values are not well-defined in the ordinary sense.  

In this work, the BVP is reformulated  by replacing the continuity of the   boundary trace mapping with measurability.  Instead of  using variational methods alone,  the reformulation  of the BVP derives also from  Cameron-Martin space techniques.
The reformulation essentially returns the study of irregular white noise loads  to  study of   $L^2$-loads.   

Admissibility of the reformulation is demonstrated  by showing that  usual  finite element approximations of the BVP  with  discretized white  noise loads converge to the solution of the   reformulated  problem.   For Neumann and Robin  BVPs, the finite dimensional approximations have been utilized before. However,  here also the infinite-dimensional limit is considered.


\end{abstract}

\section{Introduction}
\label{intro}

Let $D$ be a bounded Lipschitz domain in $\mathbb R^d$, where the dimension $d\geq 2$, 
and denote with $\partial _n$ the normal derivative at the boundary of $D$. We study  a stochastic counterpart of a  boundary value problem (BVP) 
\begin{eqnarray}
\begin{split} 
  \left\{
  \begin{array}{r l l}
- \Delta u  + \lambda u   &=  f   & {\rm   \,  in\,  } D\\
  Bu  &=  0  & {\rm  \, on\,  } \partial D,
 \end{array}
 \right.
 \end{split}\label{eq:bvp1}
 \end{eqnarray}
 where $f\in H^{-1}(D)$ is replaced with  the Gaussian zero mean  white noise $\dot W$  on $D$ and  
   the boundary operator $B$ is  either of the  Dirichlet type ($Bu=u|_{\partial D}$ ), Neumann type ($Bu= 
   \partial_n u |_{\partial D}$), or the  Robin type ($Bu=  \partial_ n  u |_{\partial D } +\beta u |_ {\partial D}$ for $\beta \in \mathbb R$).  The constant $\lambda$ in 
   \eqref{eq:bvp1} is  positive for simplicity.

 The study of  stochastic elliptic  boundary value problems  \eqref{eq:bvp1}
initiated   from  works of Walsh  \cite{walsh2,walsh}  who considered solvability of the Poisson equation with zero Dirichlet boundary  condition and the white noise source. Walsh  studied the  very  weak formulation  of   \eqref{eq:bvp1}  in   the sense of generalized functions, that is,  distributions.  

In the case of  the homogeneous Dirichlet boundary condition $Bu= u|_{\partial D}=0$ and  the  white noise load $f= \dot W$,  
 the existence and uniqueness of  pathwise
  continuous solution of  \eqref{eq:bvp1}  is well-known for $d=1,2,3$, even for nonlinear equations  by results of Buckdahn and Pardoux \cite{pardoux}. The corresponding Neumann and  Robin problems are less extensively studied, although there are  numerous studies on  elliptic   BVPs with more    regular  deterministic loads. This leaves a gap in the literature which appears  for example in connection with Bayesian statistical inverse problems, where solutions of  stochastic BVPs serve as priors \cite{MR3209311}.  The aim of this work is to provide  a rigorous description of the stochastic BVP with the white noise load that utilizes both the stochastic nature of the problem and the existing literature 
on  more regular  problems. 
  
  The main difference between stochastic Dirichlet and Neumann problems is the specification of the solution space. In  \cite{pardoux},  one seeks  a stochastic field $X$ with continuous realizations that satisfies $-\Delta X +\lambda X= \dot W$ in the sense of distributions, and $X|_{\partial D}=0$.   In corresponding  Neumann and Robin problems 
  the normal derivative at the boundary is not well-defined when only  continuity of the  realizations has been verified, which is the main obstacle  for  formulating the problem.   Indeed, the weak definition  of the (co)normal derivative $\partial_n  u$ at the boundary $\partial D$ of the variational  solution $u$ of 
  \eqref{eq:bvp1}   requires  that functionals  
 \begin{equation}\label{eq:normald}
\int_{\partial D}  \phi \partial _n  u \, d \sigma 
:= \int _D \nabla  u \cdot \nabla \phi  +  \lambda u \phi  - f \phi \, dx  
\end{equation}
are well-defined  for  all $\phi$ in a suitable function space $H$,  which e.g.  is satisfied when     
  $u\in H^1(D)=:H $ and  $f \in L^2(D)$  (see for example  \cite{mclean}).

There are several studies on how  to extend an elliptic BVP to   irregular loads or irregular  boundary data.  In     \cite{babuska2} and references therein,   BVPs are taken to be deterministic with  no loads  but  highly irregular boundary values. Obviously, the above problem 
  can be cast in such a form. For smooth boundaries, the several proposed extensions in \cite{babuska2} work 
 nicely but for polygonal domains turn out to be problematic.  A similar theme can be found in 
\cite{benfatto}.   Rozanov \cite{rozanov} treats random fields as  Hilbert space processes, and applies theory of distributions in defining the boundary traces for $C^2$-smooth boundaries. 
Smoothness of boundaries benefits the definition of  distributions on the boundary. 
 An attempt to solve  the Neumann boundary value problem with the help of Lax-Milgram theorem  is made in \cite{franklin}. However, the paper does not take into account  that some of the stochastically integrated  functions   are anticipating  which suggests that correct formulation would involve   multidimensional  Skorohod integrals.  Also the interpretation of the normal derivative is left  vague.  A correct formulation with more regular loads can be found  e.g. in  \cite{babuska,swap},  but it is  clear that the  white noise  loads do not fulfill the required conditions.   The work in 
 \cite{franklin} can be appreciated from the point of view of  a more pragmatic question, which asks whether the white noise could be approximated  by more regular stochastic fields  (in the sense of an existing limit).   
 
  The approximations of white  noise  in Dirichlet problems are often carried out together with   finite element methods  \cite{allen,ben,cao,cao2,franklin,swap,zhang}.  Also, the convergence of approximative solutions has been verified \cite{allen,ben,cao,cao2,swap,zhang}.

  In \cite{pardoux}, the homogeneous Dirichlet solution is acquired by  replacing  the BVP  with a Hammerstein integral equation.  A similar integral equation could be written in the  Neumann or Robin case (see \cite{caoscat})  by updating  the Dirichlet  Green's function $G(x,y)$ with a correct  boundary value. In the linear case, the  conjectured integral equation would  be    $X(x)  + \lambda  \int_D  G(x,y ) X(y)  dy =  \int_D  G(x,y ) dW_y $ for a stochastic field $X$  with a.s. continuous realizations,  where $ d W_y$ represents  multidimensional It\={o} integral.
 This is referred to as  the mild form of the problem. However, it is not clear  whether the realizations of $X$  would fulfill the boundary condition $\partial_n X |_{\partial  D}=0$ in any other  than mild sense. For smooth domains,  a partial answer can be found in \cite{babuska2} for  the description of the BVP, where such a formulation is compared to another generalization of  irregular boundary values.

We proceed in different direction than in \cite{babuska2,benfatto}.    
 Instead of trying to stretch  the definition of the differentiability, we stretch the definition of the boundary  trace with measure theoretic methods.      Indeed, replacing $ f$ in \eqref{eq:normald}  with $L^2$-approximations of the white noise hints  that a rigorous definition of  the normal derivative  of $X$  might not call for  continuity of the linear forms \eqref{eq:normald}  on $H^1(D)$  but  only  measurability.  Similar phenomenon appears in the variational formulations of  BVPs with different boundary conditions.  For  $f\in L^2(D)$, the variational form of  the homogeneous  Dirichlet BVP is to find $u\in H^1_0(D)$ that satisfies
   \begin{eqnarray}\label{eq:varsol}
\int  _D f(x)   \psi  (x)  d x     =   \int _D  \nabla  u (x)  \cdot \nabla \psi (x) + \lambda  u (x) \psi (x)  dx 
   \end{eqnarray}
 for all $\psi \in H^1_0(D)$.     Replacing  $f$ with regular approximations of white  noise hints again   to   measurability of the linear forms.  Indeed, in the case of  homogeneous Dirichlet problem, such approximative
   variational  solutions  are known  to converge in $L^2(\Omega, \Sigma,P; L^2(D))$-norm   to the correct solution   \cite{cao}.  The corresponding limit of the variational equations,  when refining white  noise approximations is then 
   \begin{eqnarray}
 \int _D \psi(x) d W_x  = \lim _{n\rightarrow \infty  } \int \nabla  X_n  \cdot \nabla\psi  +  
 \lambda X_n  \psi dx  
\end{eqnarray}
for every $\psi \in H^1_0 (D)$, where $ d W_x$ represents  multidimensional It\=o integral and 
$X_n$ are variational solutions of \eqref{eq:varsol} with the  approximated white noise.

     The present paper contributes in this area by  giving explicit formula  for the normal derivative  of the solution of \eqref{eq:bvp1}  as a measurable mapping (see Definition \ref{def:new}).   Instead of tackling directly the variational formulations of general BVP or trying to interpret the normal derivative in distributional sense, we 
   reformulate the irregular elliptic BVP  so that existing results for more regular   elliptic BVP can be  easily utilized. The approach also avoids the need to provide new estimates for the     corresponding Green's functions, as is often the case in mild formulations. For example, continuity of the solution of two and three dimensional   Neumann and Robin    problems  with the  white noise  load follows from the  regularity of  the  deterministic problem via    well-known Gaussian arguments. Moreover,  unique solvability of high-dimensional problems is    also guaranteed.
   
The reformulation of BVP involves Cameron-Martin space techniques. The main 
tool is  the method of extending  continuous linear mappings  $L$ on the Cameron-Martin space of a Gaussian field  $X\sim \mathcal N(0, C_X)$ into measurable linear mappings $\widehat L$ on the sample space of $X$ (see e.g. \cite{bog}).  Measurable linear extensions are applied   in defining the boundary operators for the Gaussian fields.  T   
  
In order to demonstrate  admissibility of the reformulation, we show that 
finite element approximations $X_n$  of the solution $X$ converge to the solution of 
the problem. The proof reduces essentially to a one-liner \eqref{eq:oneliner}, even for high-dimensional problems.    

 The main approach  to finite element methods (FEM) with irregular stochastic loads   was introduced in \cite{allen}, 
 where the stochastic load $f$ is first approximated by a spatially piecewise constant function, and then the ordinary FEM is applied (see 
 also  \cite{cao,franklin}).  However,  even in 1D case the solutions of \eqref{eq:bvp1} with the white noise load are  not regular  enough for 
 standard pathwise error methods \cite{allen}.   The convergence of FEM  approximations is therefore recast as a question of   convergence of random variables, where several other modes of convergence are available besides 
 to pathwise convergence.       From previous studies  \cite{cao} it is known that  random fields $X_n$ converges to $X$  in   norm   \begin{equation}\label{eq:norm}
 \Vert X\Vert : = \left( E \Vert X \Vert_{L^2(D)}^2\right)^\frac{1}{2}
 \end{equation} 
 for  2D  Dirichlet problem.   In  \cite{cao2},   3D  case on a  convex smooth domain is considered.   Also,  estimates for  the speed of convergence are known \cite{cao}.     
  Other similar works are \cite{allen,qi,franklin,zhang}. More regular loads are considered in  \cite{babuska,swap}. 
     
  We improve the previous  results by replacing $L^2(D)$ with 
$C(D)$ and giving generalization to cases of Neumann and Robin boundary data.
The cases of Neumann and Robin boundary conditions are  new.

  The contents of this paper is the following. 
  In Section  \ref{sec:1} we recall  known results  
  about Gaussian random variables and their linear  transformations.
   In Section \ref{sec:mbo} we define the measurable boundary 
   trace  and measurable normal derivatives (see Definition \ref{def:new}). In Section \ref{sec:4} we formulate the  BVP  and 
   study  its unique solvability.  In Section \ref{sec:5} the 
   regularity of the solutions is considered. In Section \ref{sec:6}
    the finite element approximations are studied.
      
 \section{Measure theoretic preliminaries}\label{sec:1}

Let $(\Omega,\Sigma,P)$ denote a complete probability space. We make a standing assumption that 
all random variables are defined on $(\Omega,\Sigma,P)$. Moreover, all appropriate  function spaces appearing below 
are endowed with their Borel $\sigma$-algebras. We will denote with 
$H^s(D)$, $s\in\mathbb R$,  the usual Sobolev spaces  on $D$ and 
with $H^s_0(D)$ the usual closure of compactly supported smooth functions on 
$D$ (see e.g.  \cite{mclean}).

In this work, we extensively use the theory of  Gaussian function-valued random variables  and their
linear functionals. As an introduction to present ideology, we recall the basic  definitions in the case of white noise.  
 
 Let  $\mathcal B(D)$ denote the Borel sets  of a bounded  Lipschitz domain $D\subset\mathbb R^d$.
    Recall, that  $\dot W$ is the white noise on $D$   if  $\{\dot W(A): A \in \mathcal B(D)\}$ are  Gaussian random variables with zero mean and covariance  $\mathbb E \dot W(A) \dot W(B)  = |A\cap B|$, where $| A|$ denotes the Lebesgue's measure of   $A$, and $\dot W(A\cup B)=\dot W(A) +\dot W(B)$ a.s. for disjoint $A$ and $B \in \mathcal B(D)$.   A common way to construct functionals   $\dot W(A)$ is through stochastic integrals 
   $$
   \dot W(A) = \int 1_A(x) dW_x,  
   $$
  with respect to $d$-dimensional  Wiener field $W_x$, which is  a Gaussian field with zero mean and covariance 
  $\mathbb E W_x W _y = \min(x_1,y_1)\cdots \min (x_d,y_d)$ for  all   $ x=(x_1,\dots,x_d), y=(y_1,\dots, y_d)\in \mathbb R^d$. 
  The It\=o isometry allows us to replace  characteristic functions $1_A$ of Borel sets $A\in\mathcal B(D)$   by  functions  $\phi \in L^2(D)$, and hence 
  define white noise  functionals 
  \begin{equation}\label{wn}
  \dot W(\phi):= \int_ D \phi(x) dW_x
  \end{equation}
as  Gaussian random variables with zero mean and variance $ \mathbb E \dot W(\phi)^2 = \Vert \phi\Vert_{L^2(D)}^2$.

Instead of considering  solutions for the elliptic 
boundary value problem as stochastic fields, we take the more 
general approach by considering solutions (and the white noise)  as   Banach space valued 
random variables.  Let us recall some definitions (e.g.  \cite{bog,bogd}).

 Let $\mathbb B$ be a  separable Banach space.  A mapping $X:\Omega \rightarrow \mathbb B$ is   a $\mathbb B$-valued random variable   if $X^{-1}(A)\in  \Sigma$ for   all Borel sets $A\subset \mathbb B$.  Denote $\mu_X=P\circ X^{-1}$ the image 
  measure of $X$ on $\mathbb B$.  Let $\mathbb B^*$ denote the topological 
  dual of $\mathbb B$ and $\left < \cdot  ,\cdot \right >_{\mathbb B,\mathbb B^*}$ denote the duality.          A  $\mathbb B$-valued random variable   $X$ is called Gaussian if $\langle X,b^* \rangle_{\mathbb B, \mathbb B^*} $ is  Gaussian for all $b^* \in\mathbb  B^*$.

 For notational simplicity, we focus on reflexive $\mathbb B$. 
In the case of reflexive $\mathbb B$, we denote 
with $m\in \mathbb B$ the mean of  $X$ i.e.
$$
\langle m ,  b^* \rangle_{\mathbb B, \mathbb B^*} = \mathbb E   \langle X, b^* \rangle _{\mathbb B, \mathbb B^*}   
$$
  for all $b^* \in \mathbb B^*$ and with $C_X: \mathbb B^*\rightarrow \mathbb B $ the covariance operator of $X$   i.e. 
  $$
\langle C_X b^* , b^* \rangle _{\mathbb B, \mathbb B^*} = \mathbb E    \langle X-m,b^*\rangle_{\mathbb B, \mathbb B^*}   \langle X-m,b^* \rangle_{\mathbb B, \mathbb B^*}         $$
for all $b^* \in \mathbb B^*$.

Next, we recall that the  white noise $\dot W$ is $H^{-d/2-\epsilon}(D)$-valued  Gaussian 
random variable for any $\epsilon>0$. Indeed,  almost sure realization properties of the $d$-dimensional white noise can be derived by using random functionals  $\dot W(\phi)$. The definition of the stochastic integral helps in  identifying realizations of  the white noise as  weak  derivatives of  realizations of the Wiener field. Then the random functional $\dot W (\phi)$ can be   identified with  the linear functional $\langle \dot W, \phi\rangle$ between  a distribution  $\dot W$ and a test function $\phi \in C^\infty_0(D)$.  It is an easy task to apply duality to  study   Sobolev norms   $ \Vert \dot W\Vert_{H^{-s}(D)} = (\sum_{k=1}^\infty  |\langle \dot W , \phi_k\rangle_{H^{-s} (D), H^{s} (D)}|^2)^\frac{1}{2}$ where   $\phi_k\in C^\infty _0(D)$  form an orthonormal basis  in  $H^{s}_0(D)$ and  $s\geq 0$.  Recall, that orthonormal basis can be chosen from a dense set, and the  dual of  $H^{-s}(D)$ can be identified with $H^s_0(D)$ (see Theorem 3.30 in \cite{mclean}).  In particular,   $\dot W $  belongs a.s.  to the Sobolev space $H^{-d/2-\epsilon } (D)$ for  any $\epsilon>0$, since  the series of  variances $\sum_{k=1}^\infty  \Vert \phi_k\Vert ^2_{L_2(D)}$ converges (see 
 \cite{kahane}, Theorem  2 in Chapter. 3)  by  Maurin's theorem (e.g.  \cite{MR0162126}). Similarly, $\dot W$ belongs  to  $H^{-d/2}(D)$ with probability zero. 

  The measurability  of white  noise can  be checked by the well-known Pettis' measurability theorem,  which says that  a $\mathbb B$-valued mapping is a $\mathbb B$-valued random variable, if it is weakly measurable i.e. mappings $\langle X,b^* \rangle$ are random variables for all $b^* \in\mathbb 
B ^*$. Hence,  $\dot W$ is $H^{-d/2-\epsilon}(D)$-valued  Gaussian 
random variable for any $\epsilon>0$.

 White  noise $\dot W$  has mean zero and
 identity as the covariance operator.

\begin{definition}\label{def:ML}
Let $\mathbb B$ be a separable reflexive Banach space and let $X$ be a Gaussian  $\mathbb B$-valued zero mean random variable whose  covariance operator  $C_X$
 is nontrivial.
Set
$$
\Vert b^* \Vert_{\mu_X} := \sqrt{\langle C_X b^* , b^* \rangle_{\mathbb B, \mathbb B^*} }  
$$
for all $b^* \in \mathbb B^*$ and denote with 
$
\mathbb B^*_{\mu_X}
$ the closure of  $\mathbb B^*$ in  norm the $\Vert \cdot \Vert _{\mu_X}$.
\end{definition}    
It is well-known  that the 
elements of $\mathbb B^*_{\mu_X}$ can be identified with  $\mu_X$-measurable linear functionals on 
$\mathbb  B$.   More precisely,  a functional    on $\mathbb B $ is  a $\mu$-measurable linear 
functional if  it is  $\mu-$measurable and it has a version that  is linear on a  linear subspace  of  full $\mu$-measure. The measurability of  $\widehat h  \in \mathbb B^*_{\mu_X}$ can be seen as follows. For every $\widehat h \in \mathbb B^*_{\mu_X}$ there exists 
a sequence $\{b^*_k \}\subset \mathbb  B^*$ so that  $\lim_{k\rightarrow \infty}
 b^*_k= \widehat  h$ in $\mathbb B _{\mu_X}$. But then  the linear functionals
 $$
b\mapsto \langle b, b^* _k\rangle_{\mathbb B, \mathbb B^*},  k \in\mathbb N, 
$$
form a Cauchy sequence in $L^2(\mu_X)$. By taking a suitable subsequence, we obtain 
$\mu_X$-a.s. limit  
\begin{equation}\label{eq:mfun}
b\mapsto \lim_{j\rightarrow \infty}  \langle b, b^* _{k_j}\rangle_{\mathbb B, \mathbb B^*}   =: 
  \widehat {h} (b) 
\end{equation}
and 
$$
\Vert \widehat h - b_{k_j}\Vert^2_{\mu_X}= \mathbb E  \left(  \widehat h(X) - \langle X, b_{k_j}^* \rangle _{\mathbb B, \mathbb B^*} \right)^2.  
$$
Each $ \widehat {h}\in \mathbb B^*_{\mu_X}$ defines a measurable functional \eqref{eq:mfun} which are linear on a full measure linear subspace (for details, 
see Theorem 2.10.9  and Theorem 3.2.3 in \cite{bog}). We summarize the 
above facts in the next lemma.
\begin{lemma}\label{remark:gaussian}
The elements  $\widehat  h$ of  $ \mathbb B^*_{\mu_X}$ can be identified with 
$\mu_X$-measurable linear functionals $b\mapsto \widehat h(b)$ that are 
Gaussian zero mean random variables on  the Lebesgue's completion of the  probability space $(\mathbb B, \mathcal  B(\mathbb B), \mu_X)$.
 Moreover, the   covariance  
$$
(\widehat h, \widehat g)_{\mu_X}  := \left(  \int  \widehat  h (b) \widehat g (b) \mu_X(db) \right),
 $$ 
where $\widehat h,\widehat g\in  \mathbb B^*_{\mu_X}$,  defines an inner product on 
$\mathbb B^*_{\mu_X}$ and $(\widehat h, \widehat h)_{\mu_X}=\Vert \widehat h\Vert _{\mu_X}^2$.
\end{lemma}

 The difference between 
just measurable and a measurable linear functional is that  the versions of  measurable linear functionals are not allowed to be 
modified on arbitrary  null sets but only those null sets  that will not destroy the linearity.

 Especially, the mapping 
$$
B^*_{\mu_X}\times B^* \ni (\widehat h, b_2^*)\mapsto \mathbb E  \widehat h (X)   \langle X,  {b_2^*}\rangle _{\mathbb B, \mathbb B^*}
 $$  
is  bilinear, and  by Fernique's theorem  (see e.g. \cite{bog})  bounded in the sense that 
\begin{equation}\label{eq:fern}
 |\mathbb E  \widehat h  (X)   \langle X,  {b^*}\rangle _{\mathbb B, \mathbb B^*}|\leq 
  \left (  \mathbb  E \widehat h (X)^2\right )^\frac{1}{2} 
  \left( \mathbb E  \Vert X \Vert_{\mathbb B}^2  \Vert  b^*\Vert_ {\mathbb B^*}^2 \right)^\frac{1}{2}\leq 
 C \Vert \widehat h \Vert_{\mu_X} \Vert  b^*\Vert_ {\mathbb B^*} . 
\end{equation}
\begin{remark}\label{rem:1}
By 
\eqref{eq:fern} and reflexivity of $\mathbb B$, we may 
 extend the
covariance operator  $C_X: \mathbb B^* \rightarrow \mathbb B$  as a continuous mapping 
from $\mathbb B^*_{\mu_X}$ to  $\mathbb B$, and we continue to denote the extension with 
$C_X$ i.e 
\begin{equation}\label{eq:rem1}
\langle C_ X  \widehat {h}, b^*\rangle_{\mathbb B, \mathbb B^*} = 
\mathbb E    \widehat  {h} (X)   \langle X,  {b^*}\rangle _{\mathbb B, \mathbb B^*}
\end{equation}
for all $  \widehat h \in B^*_{\mu_X}$ and $ b^*\in B^*$.
\end{remark}
   \begin{definition}
Let $X, C_X, \mathbb B $, and $\mathbb B^*_{\mu_X}$  be as in Definition \ref{def:ML} and 
extend $C_X$ as in Remark \ref{rem:1}.  The 
\textit{Cameron-Martin space of $X$}  is the set 
$$
H_{\mu_X}= C_X  (\mathbb B^*_{\mu_X})
$$
equipped with the inner product 
\begin{equation}\label{eq:ip}
(h,g)_{H_{\mu_X}}=  \int_ B  \widehat h (b)   \widehat g (b)    \mu_X(db),
\end{equation}
where  for all  $h\in H_{\mu_X}$ the notation   $\widehat h $ means  such a vector in  $\mathbb B^*_{\mu_X} $  that  $C_X \widehat h  =h$.
 The corresponding  inner product norm is denoted with 
$ \Vert  h\Vert _{H_{\mu_X}} $.
\end{definition}

\begin{remark}\label{remark:separable}
The Cameron-Martin space $H_{\mu_X}$ and the space of measurable linear 
functionals $\mathbb B^*_{\mu_X}$  are  separable Hilbert spaces for 
all Gaussian random variables $X$ that have values in separable Banach spaces
(see Theorem 3.2.7 in \cite{bog}). 
\end{remark}

\begin{remark}\label{rem:2}
The covariance operator $C_X: \mathbb B^*_{\mu_X}\rightarrow H_{\mu_X}$ is an 
isometric isomorphism. From the inner product \eqref{eq:ip}, we derive the bilinear form 
$$
\langle h,\widehat g\rangle_{H_{\mu_X},  B^*_{\mu_X}}= \langle  C_X \widehat h,\widehat g\rangle_{H_{\mu_X},  B^*_{\mu_X}}:=  \int_ B  \widehat h (b)   \widehat g (b)    \mu_X(db) ,
$$ 
and thus identify $\mathbb B^*_{\mu_X}$ as the dual space of the Cameron-Martin space.
By   \eqref{eq:mfun} and \eqref{eq:rem1}, 
\begin{equation}\label{eq:trans}
\langle h,\widehat g\rangle_{H_{\mu_X},\mathbb B^*_{\mu_X}} =  \widehat g( C_X\widehat h) 
 = \widehat g (h),
\end{equation}
for proper linear versions of  $b\mapsto \widehat g(b)$ since the Cameron-Martin space 
is contained in every linear subspace of full measure (see \cite{bog}, Theorem 2.4.7). By density, we may always choose an orthonormal basis $\{ \widehat e_k\}$  of $\mathbb B_{\mu_X}^*$ that consists of  functions in $\mathbb B^*$ and the corresponding image $C_X \widehat e_k \subset 
 \mathbb B $ is 
an orthonormal basis of the Cameron-Martin space
\end{remark}

We  recall that  the Cameron-Martin space of  $\dot W$ is $L^2(D)$. In general, 
the Cameron-Martin space  of a $\mathbb B$-valued random variable is separable and the Cameron-Martin space  does not depend on the sample space of  the $\mathbb B$-valued random variable $X$ (see Theorem  3.2.7  and  Lemma 3.2.2 in \cite{bog}). Moreover, since the zero-mean Gaussian  $X$ has values in  $\mathbb B$, then  
  the inclusion mapping of the  Cameron-Martin space into  $\mathbb B $ is Hilbert-Schmidt (see \cite{bog}, Corollary 3.5.11).

 Let us  recall the definition of measurable 
linear operator in our setting (see  Definition 3.7.1  in  \cite{bog} for 
a more general formulation).

\begin{definition}
Let $\mathbb B_1,\mathbb B_2$ be  separable Banach spaces equipped with their Borel $\sigma$-algebras  and  let $\mu$  be a Borel probability measure 
on $\mathbb B_1$. A mapping $T: \mathbb B_1\rightarrow \mathbb B_2$ is a \textit{$\mu$-measurable 
linear operator} if there exists a  linear   mapping $S:\mathbb B_1\rightarrow \mathbb B_2$ such that  $S$ is $\mu$-measurable and $S=T$ $\mu$-a.e. The linear mapping 
$S$ is called \textit{a proper linear version of $T$}.
\end{definition}

In the case of Gaussian measures, there is a close relationship between
 measurable linear operators and Gaussian random series (see 
 \cite{bogd}, Theorem 1.4.5 and Corollary 1.4.6-7). In the next theorem, we explicitly 
 state the form of the measurable linear operators (the result is a minor modification of 
 \cite{bogd}, Corollary 1.4.6). 
 
\begin{theorem}\label{tsir}
Let $X $ be a zero mean Gaussian random variable  on a separable reflexive  Banach space  $\mathbb B_1$, let 
$H_{\mu_X}$ denote the  Cameron-Martin space of $X$ and let   $\{\widehat e_k\}_{k=1}^\infty \subset \mathbb B^*$
denotes an orthonormal basis of $\mathbb B^*_{\mu_X}$.  If  $T$ is a  continuous linear mapping from $H_{\mu_X}$ into a separable Hilbert  space $  \mathbb H$, then
\begin{equation}\label{eq:tsir}
\widehat  T (b)    = \sum_{k=1}^\infty    \widehat  e_k  (b)   T  C_X \widehat e_k
\end{equation}
defines a $\mu_X$-measurable linear operator $\widehat  T : \mathbb B_1\rightarrow \mathbb B_2$ for any separable Hilbert space 
 $\mathbb B_2$ such that the inclusion mapping  $\mathbb H\hookrightarrow  \mathbb B_2$ is Hilbert-Schmidt. 
 
Moreover,  if $R:\mathbb B_1 \rightarrow \mathbb B_2$ is  a  $\mu_X$-measurable linear operator whose proper linear version $R_0$ coincides with $T$ on $H_{\mu_X}$, then 
 $R=\widehat T$ $\mu_X$-a.e. 

 \end{theorem}

\begin{proof}

The proof follows  the well-known lines  (e.g. \cite{bog}, Theorem 3.7.6).  For completeness, we provide a sketch of the proof. 

First,  we verify that any two proper linear measurable mappings   $  S_1, S_2: 
  \mathbb B_1\rightarrow \mathbb B_1$  that  
   coincide on the Cameron-Martin space of $X$  coincide $\mu_X$-a.s. on $\mathbb B_1$. 
   
The space  $\mathbb B_2$ is  separable and reflexive. Therefore, there exists a 
   countable subset $G$ of  $\mathbb B_2^*$ that separates the points
   of $\mathbb B_2$.  We only need to verify 
   that
   $$
   \langle  S_1 x, b^*\rangle=  \langle  S_2 x, b^*\rangle
   $$
 $\mu_X$-almost surely  for every $b^*\in G$. 
 
 The both mappings $x\mapsto  \langle  S_i x, b^*\rangle$, $i=1,2$, 
   are linear and measurable functionals, and they coincide on the Cameron-Martin 
   space. Therefore, they coincide $\mu_X$-almost surely (see Theorem 2.10.7 in \cite{bog}).
   Hence the two proper linear measurable mappings coincide $\mu_X$-almost surely.

Next, we verify that the series is $\mu_X$-almost surely convergent. 
 Since the inclusion mapping of $\mathbb H$ into $\mathbb B_2$ is Hilbert-Schmidt, also 
 the mapping $T: H_{\mu_X}\rightarrow \mathbb B_2$ is Hilbert-Schmidt.  By definition, 
 $$
 \sum_{k=1}^\infty  \Vert T  (C _X \widehat e_k) \Vert_{\mathbb B_2}^2 < \infty 
 $$
 for any orthonormal basis $\{C_X \widehat e_k\}$ of $H_{\mu_X}$.
  
    Then  the random series \eqref{eq:tsir} is  $\mu_X$-a.s. convergent, since the sum of  variances $\sum_{k=1}^\infty  \Vert T C_X e_k\Vert ^2_{\mathbb B_2}$ converges   (e.g \cite{kahane}, Theorem  2 in Chapter 3). Indeed, the  coefficient $b\mapsto \widehat e_k(b)$ of $TC_Xe_k$ in the series   \eqref{eq:tsir} are  normal random 
  variables on the Lebesgue's completion  of the  probability space $(\mathbb B_1, \mathcal B (\mathbb B_1), \mu_X)$  by Lemma \ref{remark:gaussian}. Moreover, they are independent  since $\{\widehat e_k\}$ is orthonormal.

 The set $L\subset \mathbb B_1$, where the series \eqref{eq:tsir} converges is  $\mu_X$-measurable  linear subspace of  full measure. Moreover, $\widehat T : L \rightarrow \mathbb B_2$ is linear  since  $T$ is linear.  We extend $\widehat T$ linearly onto $\mathbb B_1$ by taking such a linear subspace $M$ of $\mathbb B_1$ that 
 $\mathbb B_1$ is a direct algebraic sum of $L$ and $M$ and defining 
 $\widehat T(b +b'):= \widehat Tb $ for   $b\in L$ and $b'\in M$.  Since the convergence holds a.s.,   the mapping  $\widehat T:\mathbb B_1\rightarrow \mathbb B_2$ is  measurable with respect to the Lebesgue's completion of the Borel $\sigma$-algebra of $\mathbb B_1$.  
  %
  
\qed\end{proof}

\begin{definition}
Let $T$ and $\widehat T$ be as in Theorem \ref{tsir}.
The mapping $\widehat T$  is called  \textit{measurable linear extension of 
$T$}. 
\end{definition}

\begin{corollary}\label{cor:multi}
Let the assumptions in Theorem \ref{tsir} hold. 
The following claims hold  for measurable linear extension $\widehat T$ of $T:H_{\mu_X}\rightarrow \mathbb H\hookrightarrow \mathbb B_2$.
 \begin{enumerate}
 \item[(i)] The set  $T(  H_{\mu_X} )$ coincides with the  Cameron-Martin space 
of $\widehat  T X$  and   the mapping   $T: H_{\mu_X} \rightarrow H_{\mu_{\widehat T X}}$ has  unit norm. 

 \item[(ii)]  (Measurable transpose)  Let  $T^*:\mathbb B_2  ^*\rightarrow  ( \mathbb  B_1^*) _{\mu_X} $ denote the transpose of $T$ and let $b^*\in \mathbb B_2^*$. Then 
 $$\langle \widehat T b , b^*\rangle_{\mathbb B_2, \mathbb B_2^*} = 
   T^* b^* ( b)$$ for $ \mu_X$-a.e. $b\in \mathbb B_1$.
 
 \item[(iii)]  (Associativity of compositions) Let $\mathbb H_2$ be a separable Hilbert space, whose inclusion into separable 
 Hilbert space $\mathbb B_3$ is Hilbert-Schmidt.  When $\widehat S: \mathbb B_2\rightarrow \mathbb B_3$  is  measurable linear extension of  the  continuous linear mapping  $S: H_{\mu_{\widehat T X}} \rightarrow \mathbb H_2$, 
 then   $$\widehat {ST} X = \widehat S ( \widehat T X )$$ almost surely. 
 \item [(iv)]  When $T$ is the identity mapping, we have 
 $$
  b = \widehat T b 
 $$
 $\mu_X$-a.e.
\end{enumerate}

\end{corollary}
\begin{proof}
(i) The characterization of the elements of the  Cameron-Martin space 
 follows as in Theorem  3.7.3  in \cite{bog}, which also shows that the
 mapping $T$ has unit norm. 
 
 (ii)  Note that the mappings $ b\mapsto \langle \widehat  Tb, b^* \rangle_{\mathbb B_2, \mathbb B_2^* }$ 
 and $ b\mapsto     T^* b^* (b) $ are measurable and have proper linear versions that 
 coincide on $H_{\mu_X}$ for  $b^*\in \mathbb B^*$ by \eqref{eq:trans}.  By Theorem 2.10.7 in \cite{bog}, these  measurable  functionals coincide $\mu_X$-a.s. 
 
 (iii) Both mappings are measurable linear operators. Indeed, $\widehat S\widehat T X$ is well-defined $P$-measurable mapping, since $\widehat S$ is $\mu_{\widehat T X}$-measurable and the set  $\{\widehat TX \in B\}$ has zero measure whenever $\mu_{\widehat TX}(B)=0$.  The linearity on full measure linear subspace follows then from the definition of 
    extension. Considering  approximating sequences of measurable linear functionals and   (ii), we obtain
  $$
 \langle  \widehat S \widehat T b, \widehat h\rangle_{\mathbb B_3, \mathbb B_3^*}  =  S^* \widehat h ( \widehat T b)= T^* 
  S^*  \widehat h(b) = \langle \widehat {ST} b,\widehat h\rangle_{\mathbb B_3, \mathbb B_3^* } 
 $$
$\mu_X$-a.e  for each $\widehat h\in B_3^*$. Taking $\widehat h$ from some 
 countable dense subset of $\mathbb B_3^*$ proves the claim. 

 (iv) See Theorem 3.5.1 in \cite{bog}.
 
\qed\end{proof} 

\section{Measurable boundary operators} \label{sec:mbo}

Theorem \ref{tsir} allows us to define the  measurable linear extensions of the boundary operators
$$
B u= u |_{\partial D}
$$
and 
$$
Bu= \partial _n u|_{\partial D}+\lambda u|_{\partial D}
$$
for $u\in H^1(D)$. Here we omit  writing out the inclusion mappings. 
 
 For simplicity,  the sample space of the boundary mapping is taken to be a scale space (for Banach scale spaces, see  \cite{scales}).  In particular, let us denote with  $
H_{sc}(\partial D), 
$  
 the closure  of $H^{-\frac{1}{2}}(\partial D)$ with respect to the norm
$$
\Vert u\Vert_{sc}:= \left( \sum_{k=1}^\infty  k^{-2} (u,f_k)_{H^{-\frac{1}{2}}(\partial D) } ^2\right)^\frac{1}{2},
$$ 
where $(f_k)$ is a fixed orthonormal basis of $H^{-\frac{1}{2}}(\partial D)$. 
 
   The choices of  sample spaces have  little effect  for the stochastic analysis.  
\begin{corollary}\label{cor:CM}
Let $X$ be an  $H^{-r}(D)$-valued Gaussian zero mean random variable for some 
$r\geq 0$. 
 Let $B:H_{\mu_X }\rightarrow H^{-\frac{1}{2}}(\partial D)$ be a continuous linear mapping and let  $(e_k)$ be an orthonormal basis of $H_{\mu_X}$.  Then
 $$
 \widehat B b  = \sum_{k=1}^\infty  \widehat e_k(b)   B e_k
 $$
belongs to $H_{sc}(\partial D)$  for $\mu_X$-a.e. $b$, the mapping   $\widehat B: 
H^{-r}(D)\rightarrow H_{sc}(\partial D)$  is a $\mu_X$-measurable linear operator, and 
$\widehat B X$ is $H_{sc}(\partial D)$-valued Gaussian random variable that has 
zero mean and covariance operator $C_{\widehat B X}$ satisfying 
$ C_{\widehat B X}  u=  B C_X B^* u  $ for all $u\in H^\frac{1}{2}(\partial D)$. 
\end{corollary}
\begin{proof}
The claim follows from Theorem \ref{tsir} after we verify that the inclusion of 
$H^{-\frac{1}{2}}(\partial D)$ into $H_{sc}(\partial D)$ is Hilbert-Schmidt.
By definition,  this follows from 
  $$
\sum_{k=1}^\infty   \Vert f_k\Vert ^2_{sc} =   \sum_{k=1}^\infty  k^{-2} <\infty.   
$$  
\qed\end{proof} 

Corollary \ref{cor:CM} allows us to define the measurable linear  extensions of 
boundary operators.
\begin{definition}\label{def:new}
  Let $X$ be $H^{-r}(D)$-valued  Gaussian zero mean random variable for some $r\geq 0$ whose 
Cameron-Martin space $H_{\mu_X}$ can be continuously included in $H^1(D)$, and let
 $(e_k)$  be an orthonormal basis of $H_{\mu_X}$.
  
The \textit{measurable trace  of  $X$} onto $\partial D$ is 
the $H_{sc}(\partial D)$-valued Gaussian zero mean  random variable 
$$
\widehat {Tr} X = \sum_{k=1}^\infty \widehat e_k(X) e_k|_{\partial D}. 
$$ 
A proper linear version of the corresponding mapping $\widehat {Tr}$ is 
called the \textit{$\mu_X$-measurable trace}.

 Assuming additionally that $\Delta u \in L^2(D)$ for all $u\in H_{\mu_X}$,  \textit{the measurable normal derivative  of $X$} at $\partial D$ is
the $H_{sc}(\partial D)$-valued Gaussian zero mean  random variable  
$$
\widehat \partial_n X = \sum_{k=1}^\infty \widehat e_k(X) \partial_n e_k ,
$$ 
where $\partial _n e_k$ denotes 
the usual conormal derivative of $e_k$. A proper linear version of the corresponding mapping $ \widehat \partial_n $ is 
called the \textit{$\mu_X$-measurable normal derivative}.
\end{definition}

Let us  now verify that the $\mu_X$-measurable trace and  $\mu_X$-measurable normal derivative are 
extensions of the usual operations. 

\begin{lemma}\label{lem:1} 
Let  $\widehat B$ be  a $\mu_X$-measurable trace or a $\mu_X$-measurable normal derivative.
\begin{enumerate}
\item[a)]  If $X$ has a.s. values in $H^1(D)$, then $\widehat B X = BX$ almost surely.
\item[b)]  If $u$ belongs to the Cameron-Martin space of $X$ and   $ \widehat B u =0$ in $H_{sc}$, then  $   B u =0   \text{ in }  H^{-\frac{1}{2}}(\partial D)$.
 \end{enumerate}  
\end{lemma}
\begin{proof}
 a)    The random  series
$$
X = \sum_{k=1}^\infty \widehat e_k(X) e_k
$$
converges in  $H^1(D)$, and we have 
$$
\widehat {Tr} X = \sum_{k=1}^\infty \widehat e_k(X) e_k|_{\partial D} = X|_{\partial D} 
$$
by continuity of the trace operator. Similar result holds for the normal derivative. 

b)  Since we are dealing with proper linear versions, we have that 
$ \widehat B =  B$ on the Cameron-Martin space, and hence 
  $Bu=0$ in $H_{sc}$. Moreover,
  $H^{-\frac{1}{2}}(\partial D)$ is 
dense in $H_{sc}$.  
 \qed\end{proof} 
 
\begin{remark}
Take $X$   to be $H^{-r}(D)$-valued  Gaussian zero mean random variable.
Since the dual  $H^{r}(D)$  of the sample space $H^{-r}(D)$  of $X$ is dense in 
$(H^{r}(D))_{\mu_X}$  and $C_X:(H^{r}(D))_{\mu_X} \rightarrow H_{\mu_X} $ is 
an isometry by Remark \ref{rem:2},  the orthonormal basis $(e_k)$ of $H_{\mu_X}$ can be always chosen  
 so that  $\widehat e_k $ is from the dual $H^r(D)$ of the sample space $H^{-r}(D)$ of $X$.
Then 
$$
\widehat e_k (X)=\langle X, \widehat e_k \rangle _{H^{-r}(D), H^r(D)},
$$ 
by \eqref{eq:mfun} which is  notationally simpler choice of a proper linear basis.
\end{remark}

\section{Existence and uniqueness of  the solution} \label{sec:4}

A rough description  of  the  Cameron-Martin space of  a Gaussian random 
variable $X$  leads to a  crude  idea of the  regularity of  $X$.  
\begin{lemma}\label{lem:2}
Let $X$ be a zero mean Gaussian  $H^{-r}(D)$-valued random variable,  whose Cameron-Martin space  $H(\mu_X)$ can be  continuously imbedded into 
$H^1(D)$.  Then  $X$ has realizations in $H^{1-d/2-\delta}(D)$ a.s.  for each 
$\delta>0$.
\end{lemma}
\begin{proof}
A  Lipschitz domain is an extension domain (see Theorem A.4 in  \cite{mclean}). 
Hence, we may apply Maurin's theorem, which  tells that the inclusion of the $H^1(D)$ into 
 into $H^{-r}(D)$ is    Hilbert-Schmidt  \cite{MR0162126}. Hence, also the Cameron-Martin space 
 can be imbedded into $H^{-r}(D)$ by a Hilbert-Schmidt mapping. Hence, $X$ belongs 
 a.s. in $H^{-r}$
\qed\end{proof}

We are now ready to formulate the measurable form of the BVP. 
\begin{theorem}\label{th:ex}
Let $D\subset \mathbb R^d$ be a bounded Lipschitz domain,   $\dot W$ be  the Gaussian  white noise on $D$, $r>d/2-1$, and  $\lambda, \beta >0$. 

There exists a  pathwise unique Gaussian zero mean  $H^{-r}(D)$-valued random field  $X$  that satisfies  the Dirichlet (or Neumann or Robin) BVP in the following sense: 
\begin{enumerate}
\item  the Cameron-Martin space  $H(\mu_X)$ of $X$  can be continuously imbedded  into $H^1(D)$  and all $h\in H(\mu_X)$ satisfy $\Delta h\in L^2(D)$,
\item the field $X$ satisfies  
\begin{eqnarray}\label{eq:B}
-\Delta X + \lambda   X= \dot W
\end{eqnarray}
in the sense of generalized functions, and
\item   the field $X$ satisfies the $\mu_X$-measurable boundary condition 
\begin{eqnarray}
 \widehat {Tr}  X =0  \text{ in }  H_{sc}(\partial D) \label{eq:C}
\end{eqnarray}
in the Dirichlet case (or  
\begin{eqnarray}
 \widehat {\partial_n}  X =0  \text{ in }  H_{sc}(\partial D)  \label{eq:C2}
\end{eqnarray}
in the Neumann case,  or  
\begin{eqnarray}
 \widehat {\partial_n}  X  + \beta \widehat{Tr} X =0  \text{ in }  H_{sc}(\partial D)  \label{eq:C3}
\end{eqnarray}
in the Robin case, correspondingly).
\end{enumerate}
\end{theorem}

At a  glance, the above reformulation  may  seem  eccentric. However, it involves typical elements of 
elliptic BVPs. Namely,  the regularity of the desired solution $X$ is explicitly  specified. This is done 
by requiring that   (a) the sample space of $X$ is at least in the Sobolev space $H^{-r}(D)$, (b) $X$ has a  Gaussian distribution,  and  (c) the inclusion of the  Cameron-Martin space of $X$ into $H^1(D)$ is continuous.   The space $H^{-r}(D)$ may seem unnecessary irregular, but this is not a hindrance, since  local  and global regularity  of the solution can be further studied and refined.   On the other hand, such a weak  condition is easy to verify. 

 The Gaussianity of the solution is explicitly required in order to apply  Cameron-Martin space techniques. In particular, $\widehat B:H^{-r}(D) \rightarrow (H^{-\frac{1}{2}}(\partial D))_{sc}$ is  the     measurable linear extension of  the continuous linear operator  $B \circ I : H\rightarrow  H^{-\frac{1}{2}}(D) $ (see e.g. \cite{bog}). The restriction  (i) on the Cameron-Martin space is needed for the definition of the normal derivative.  In ordinary elliptic BVPs, the   boundary trace of $H^1(D)$-functions   is defined as a continuous linear 
extension of the trace  operator defined originally on  continuous functions.  In the same spirit,   the boundary operator   is extended from $H^1(D)$ onto  aspired solutions. However, the extended boundary operator $\widehat B$ appearing in       \eqref{eq:C}, \eqref{eq:C2}, and \eqref{eq:C3} is no longer required to be  continuous but only measurable, which makes  $\widehat B X$  well-defined generalized random field on the boundary (see Definition \ref{def:new}). 

 Another significant difference is that  $\widehat B$ depends  on the solution  $X$ through its Cameron-Martin space. However,
 it can be shown that   the extension $\widehat B$ coincide with the ordinary continuous boundary operator $B$  for $L^2$-loads,
 which are dense in negatively indexed Sobolev spaces  and $X$ contributes 
 to assigning probabilities to sets.

\begin{theorem}\label{thm:t}
Let $X$ be as in Theorem \ref{th:ex}. Then   
$$
X=\widehat T \dot W, 
 $$
where  $T:L^2(D)\rightarrow H^1(D)$ is defined by setting 
$T f :=u$, where  
\begin{eqnarray}
\begin{split} 
  \left\{
  \begin{array}{r l l}
- \Delta u  + \lambda u   &=  f   & {\rm   \,  in\,  } D\\
  Bu  &=  0  & {\rm  \, on\,  } \partial D,
 \end{array}
 \right.
 \end{split}\label{eq:bp}
 \end{eqnarray}
and  the boundary operator $Bu=u|_{\partial D}$  in the Dirichlet case or  $Bu= 
   \partial_n u |_{\partial D}$ in the Neumann case or  $Bu=  \partial_ n  u |_{\partial D } +\beta u |_ {\partial D}$  in the Robin case.

\end{theorem}
\begin{proof} [Theorem \ref{th:ex} and \ref{thm:t}]
We consider only  Robin boundary values. Other boundary conditions are handled 
similarly  (for the Neumann case, choose  e.g. $\beta =0$). For the existence of the Gaussian field, we represent  white  noise
$\dot W$ as 
  $$
  \dot W=\sum_{k=1}^\infty \widehat f_k(\dot W)  f_k,
  $$
  where $(f_k) $  is an orthonormal basis of $L^2(D)$ (see Corollary \ref{cor:multi} (iv)).  

 Now define  $T$   as in Theorem \ref{thm:t}. Such  $T$ exists according to  the well-known theory  of elliptic BVPs   (e.g. Theorem 4.11 in \cite{mclean}).
  
    Then the random series 
  $$
  X :=\sum_{k=1}  \widehat f_k(\dot W)    T f_k  
 $$
is convergent in $H^{-s}(D)$ for any $s>{d/2}$  by Theorem \ref{tsir} and   defines a zero mean Gaussian random field, whose Cameron-Martin space is  
  $T(L^2)$  equipped with the norm
  $$
 \Vert T f \Vert _{H_{\mu_X}} = \Vert f \Vert _{L^2(D) } 
  $$
  by Corollary \ref{cor:multi}.
Application of the well-known stability estimate 
$$
\Vert T f \Vert_{H^1(D)} \leq C \Vert f\Vert_{L^2(D)}
$$
shows that  the Cameron-Martin space $H_{\mu_X}$ can be continuously included in 
$H^1(D)$. This shows also that the realizations of   $X$    belong to $H^{-r}(D)$ by 
Lemma \ref{lem:2}.  Furthermore,  the  operator  $-\Delta +\lambda  $ is  continuous  on distributions, and we have    
  $$
  (-\Delta +\lambda  )X= \sum_{k=1}^\infty \widehat f_k(\dot W)  ( -\Delta +\lambda ) T f _k    =\dot W
  $$ 
  almost surely, since $ ( -\Delta +\lambda ) T f _k =f_k $ for all $k$.   Moreover, by Corollary \ref{cor:multi} 
  \begin{eqnarray*}
  \widehat{\partial_n} X + \beta \widehat{Tr} X&=&  \widehat{\partial_n}  \widehat T  \dot W + \beta \widehat{Tr}   \widehat  T  \dot W
  =
  \widehat  { \partial _n  T } \dot W +\beta \widehat{ Tr T } \dot W \\
  &=&   \sum_{k=1}^\infty  \widehat {f_k}( \dot W)       (\partial _n  + \beta Tr ) T f_k  =0
  \end{eqnarray*}
  almost surely.
  
 To prove the uniqueness, assume  that there are two solutions $X$ and $\widetilde X$ with measurable boundary 
 operators $\widehat B$   and $\widetilde B$, respectively. We show that 
 the Cameron-Martin space of $X-\widetilde X$ is  then the trivial space $\{0\}$.  Since the Cameron-Martin spaces of $X$ and $\widetilde X$ are included  continuously in $H^1(D)$,
  also the Cameron-Martin space of  $X-\widetilde X$ is included continuously in $H^1(D)$.  
  (Indeed,  by Theorem 3.3.4 in \cite{bog} it  suffices to consider  
  continuous linear forms and apply Cauchy-Schwartz inequality).
  
  By assumption,  $(-\Delta + \lambda)(X-\widetilde X)=0$ almost surely. Hence,  the same holds for 
  all functions $u\in H^1(D)$ that belong to  the Cameron-Martin space of $X-\widetilde X$
   by linearity  of $-\Delta +\lambda$. Indeed, we may consider countably many 
   continuous linear functionals that separate the points of $H^{-r-2}(D)$ and 
   apply Theorem 2.10.7 in \cite{bog}.    Moreover,  $\widehat B u  =B u= \widetilde  B u $ in $H^{-\frac{1}{2}}(\partial D)$ by Lemma \ref{lem:1}.  
  By uniqueness of the deterministic problem $u\equiv 0$.  But Cameron-Martin space 
  reduces to $\{0\}$  only   if
  $X=\widetilde X$ a.s. 
    \qed\end{proof} 
  
   \section{Improvements in global regularity}\label{sec:5}

  We first show when the  solutions of the BVP's with the white noise 
  load  are square integrable. According to the following theorem,  the  dimension of the space and regularity of the domain have the key role. 
  
  \begin{theorem}
Let $X$ and $T$ be   as in Theorem \ref{thm:t}. Then $X\in L^2(D)$ almost surely if 
and only if $T:L^2(D)\rightarrow L^2(D)$   is a Hilbert-Schmidt operator. 
 \end{theorem}
\begin{proof}
Consider an orthonormal basis $\{\phi_k\}$ of $L^2(D)$ consisting of smooth 
functions.  Then $\langle X , \phi_k\rangle _{H^{-r}, H^{r}}$ are  well-defined random variables  and we can  study finiteness of  the expectation 
$$
\mathbb E  \sum_{k=1}^\infty   \langle X, \phi_k\rangle ^2 = 
\sum_{k=1}^\infty  \mathbb E \langle \dot W, T^* \phi_k\rangle^2= 
\sum_{k=1}^\infty \Vert T^*\phi_k\Vert^2_{L^2(D)}.
$$
This shows the claim, since $T$ is a Hilbert-Schmidt operator if and only if
$T^*$ is a Hilbert-Schmidt operator.
\qed
\end{proof}

  It is often more natural to show that the realizations are actually continuous.
  Let us equip the H\"older space $C^{0,\alpha}(D) $, $0<\alpha<1$ with the usual norm 
  $$
  \Vert f\Vert_{C^{0,\alpha}}=  \sup_{x\in D}  |f(x) | + \sup_{x,y\in D, \
  x\not=y} \frac{|f(x)-f(y)|}{|x-y|^\alpha}
  $$
  \begin{theorem}
  Let $X$ and $T$ be   as in Theorem \ref{thm:t}.   If the regular solution operator
$T:L^2(D)\rightarrow C^{0,\alpha}(\overline D)$ continuosly, then 
  $X$ has a.s. continuous realizations.
  \end{theorem}
  \begin{proof}
  Let $(f_k)$ be an orthonormal basis of $L^2(D)$.  As a linear combination 
  of continuous functions, the random fields
  $$
\widehat T_N  \dot W :=   \sum_{k=1}^N \widehat f_k (\dot W) T f_k
  $$
  have  a.s. continuous realizations for all $N\in\mathbb N$. Moreover, the  a.s.  limit
  $\lim_{N\rightarrow \infty}  \widehat T_N \dot W (x)$ exists for each $x$. Indeed,   by the assumptions,  the composition of  $T$ with  pointwise evaluation $\langle \cdot, \delta_x\rangle$ is  a continuous linear functional   on
  the Cameron-Martin space of $W$. Hence,  it has a measurable linear extension
  $$
 b \mapsto   \sum_{k=1}^\infty \widehat f_k (b) T f_k(x)
  $$
   on  the sample space $H^{-r}(D)$ of the white noise.   Furthermore, the distributions of  the sequences
  $\widehat T_N \dot W (x)$ are tight on $H^{-r}(D)$ for fixed $x$. Moreover,
   $$
   \mathbb E  | \widehat T_N \dot W (x) - \widehat T_N \dot W (y) |^2 = 
    \sup_{\Vert f \Vert _{ L^2(D)}\leq 1 }   |T_N f (x) - T_N f(y)|^2 \leq  \Vert T \Vert _{L^2,C^{0,\alpha}} |x-y|^\alpha,
   $$  
   since $T_N f = T \sum_{k=1}^N \langle f,e_k\rangle e_k$ and 
   $$
   \Vert T_N f \Vert_{C^{0,\alpha}} \leq \Vert T\Vert_{L^2, C^{0,\alpha}} \Vert 
   \sum_{k=1}^N \langle f,e_k\rangle e_k\Vert _{L^2} \leq 
    \Vert T\Vert_{L^2, C^{0,\alpha}} \Vert f\Vert _{L^2}. 
   $$
   By Kolmogorov tightness criterium (e.g. see \cite{schilling} for a nice Besov space proof), the limit 
   $\widehat T \dot W =\sum_{k=1}^\infty \widehat f_k (\dot W) T f_k$   has a.s. continuous realizations.
  \qed\end{proof}

   The following familiar examples demonstrate the application of the above 
   theorem.
\begin{example}
\begin{enumerate}
\item When $D\subset \mathbb R^d$ is a bounded  Lipschitz domain, then
there exists such  $0<\alpha  <\frac{1}{2}$ that
$$
\Vert T f \Vert_{H^{1+\alpha }(D)}\leq C \Vert f\Vert _{L^2(D)}
$$
in Dirichlet   boundary value problems (see \cite{MR1600081}). 
The embedding $H^{1+\alpha }\hookrightarrow C^{0,\alpha }(\overline D)$ is continuous. Hence,
$X=\widehat T\dot W $ has a.s. continuous realizations. 

 \item When $D\subset \mathbb R^d $, $d=2,3$,   is  a bounded convex  Lipschitz domain   then
 $$
 \Vert T  f\Vert_{H^2(D)}\leq C \Vert f\Vert_{L^2(D)}
 $$ 
 for the Dirichlet problem (e.g \cite{MR775683}, Theorem 3.2.1.2).
 The embedding $H^2(D)\hookrightarrow  C^{0,\alpha }(\overline D)$ is 
 continuous for $0<\alpha <1/2$. Hence, the assumptions are satisfied and $X=\widehat T\dot  W$ has 
 a.s. continuous realizations.   
 
 \item When $D\subset \mathbb R^d$, $d=2,3$, then 
 $$
 \Vert T f \Vert_{C^{0,\alpha}}\leq C \Vert f\Vert
 $$ 
 in the case of  Robin or Neumann boundary conditions (see \cite{MR2812574}, 
 Theorem 3.14).
 \end{enumerate}
 
   \end{example}

\section{Approximations}\label{sec:6}

  At this point, it is yet unclear if the generalizations of the Neumann and Robin 
  boundary values  have any more value 
  than mathematical eccentricity.  However,  we show now that when white  noise is replaced with its  regular approximations,  the   corresponding  approximative solutions converge to the  solution of the generalized problem. This clarifies  the generalizations from a practical point of view.    In the same spirit, we study convergence of  Galerkin approximations of  high-dimensional problems  in 
  Theorem \ref{the:convergence}.   We emphasize  that such approximations are interesting, for example,  as priors in  numerical Bayesian estimation of unknown multivariable functions \cite{MR3209311}.   Finally, we consider in this section some   low-dimensional  problems as examples.

     The first convergence theorem concerns approximating the  white  noise only. 
        From various  possible approximations of  the white  noise, we first choose 
   the truncated sums in  the measurable linear  extension of the identity mapping. However, 
   the proof only requires $L^2$-convergence, and  transfers therefore  to a wider class of  approximations. 
   
   We use the generic notation $B$ for  any boundary operator appearing in 
   Theorem \ref{th:ex}. 
    \begin{theorem}
    Let $D\subset \mathbb R^d$   be a bounded Lipschitz domain,  $\dot W$ be the Gaussian white noise on $D$,     $(f_k)$   be an orthonormal basis of $L^2(D)$,  $r>d/2-1$  and 
    let $X$ be as in Theorem \ref{th:ex} .

      If   $H^1(D)$-valued random fields $X^{(m)}$, $ m\in\mathbb N$,    satisfy      \begin{eqnarray*}  
\begin{split} 
  \left\{
  \begin{array}{r l l}
- \Delta X^{(m)}  + \lambda X^{(m)}   &=  \dot W^{(m)}   & {\rm   \,  in\,  } D\\
  B X^{(m)}  &=  0  & {\rm  \, on\,  } \partial D, 
 \end{array}
 \right.
 \end{split}
 \end{eqnarray*}
  where $\dot W^{(m)}= \sum_{k=1}^m   f_k(\dot W) f_k$, then   $X^{(m)}$ converges to
$X$ in $L^2(\Omega,\Sigma,P; H^{-r}(D))$ as 
$m\rightarrow \infty$.    
\end{theorem}
\begin{proof}
The realizations of the load   $\dot W^{(m)}$ are almost surely  in $L^2(D)$. Hence,  it is easy to see that the  unique field $X^{(m)}$ exists for every $m\in\mathbb N$.
 
 Denote with $T f :=u$ the solution of 
 
 \begin{eqnarray}
\begin{split} 
  \left\{
  \begin{array}{r l l}
- \Delta u  + \lambda u   &=  f   & {\rm   \,  in\,  } D\\
    Bu =  0  & {\rm  \, on\,  } \partial D.
 \end{array}
 \right. 
 \end{split}\label{eq:bvp}
 \end{eqnarray}
  Then  $X=\widehat T \dot W$ almost surely by 
  Theorem \ref{thm:t}. Moreover,    
  $T:L^2(D)\rightarrow H^1(D)$ continuously and the mapping  
  $T: L^2(D) \rightarrow  H^{-r}(D)$ is Hilbert-Schmidt 
  (see   \cite{MR0162126}).

   Denote with $P_m:L^2(D)\rightarrow L^2(D)$ the orthogonal projection onto the  subspace spanned by  $\{f_1,\dots,f_m\}$. Then 
  $X^{(m)}= T \dot W^{(m)} =  T \widehat P_m  \dot W = \widehat{TP_m}  \dot W$ by  Corollary \ref{cor:multi}. Denote with $(g_\ell)$  the orthonormal basis of $H^{r}_0(D)$, which we identify with the dual space of $H^{-r}(D)$. By Corollary \ref{cor:multi},
  \begin{eqnarray*}
  \mathbb E  \Vert  X^{(m)} - X\Vert_{H^{-r}(D)} ^2 &=& 
  \mathbb E    \sum_{\ell =1}^\infty  \langle  X^{(m)}- X , g_\ell\rangle_{H^{-r}(D), H^{r}_0(D)}  ^2 \\
  &=&
  \sum_{\ell=1}^\infty \mathbb E   (P_m-I) T^*g_\ell  (  \dot W)   ^2 
  \\ 
  &=& 
  \sum_{\ell=1}^\infty  \Vert (P_m -I)T^* g_\ell \Vert_{L^2}^2,
 \end{eqnarray*}
  which converges to zero as $m\rightarrow \infty$  by Lebesgue's dominated 
  convergence, since    $T:  L^2(D)\rightarrow H^{-r}(D)$ is a Hilbert-Schmidt operator
   and   $$\Vert (P_m -I)T^* g_\ell\Vert_{L^2(D)}\leq \Vert T^*g_\ell\Vert_{L^2(D)}.$$
  \qed\end{proof}

 Next, we study  different approximations of the problem arising from 
 Ritz-Galerkin methods.    As a preliminary step,  we  clarify connections between  certain measurable linear forms and $L^2$-regular  approximations of the white noise. We anticipate  FEM by indexing the  Galerkin subspaces with  $h>0$, which is typically connected to the size of elements in a finite element mesh.
\begin{lemma}\label{lemma:projection}
Let $Q_h:L^2(D) \rightarrow L^2(D)$ be the orthogonal projection onto a finite-dimensional linear subspace $V_h$ of  $ L^2(D)$.  Then $\widehat Q_h \dot W\in V_h$ almost surely.  \end{lemma}
 \begin{proof}
 Let $\{e_k\}$ be an orthonormal  basis of $L^2(D)$ whose $K$ first components 
 span $V_h$. Then 
 $$
 \widehat Q_h \dot W = \sum_{k=1}^K  \widehat e_k( \dot W) e_k \in V_h
 $$
 almost surely.  
  
  \qed\end{proof}

The following theorem demonstrates how convergence of  Ritz-Galerkin approximations
 reduces to convergence of regular cases.  Here the Ritz-Galerkin approximations always include also approximations of the white noise with orthogonal projections as 
 in Lemma \ref{lemma:projection}.
We focus on Neumann and Robin problems, which have not been studied before.

\begin{theorem}\label{the:convergence}
Let $D\subset \mathbb R^d$ be a bounded Lipschitz domain,    $\dot W$ be the Gaussian white noise on $D$, $\lambda, \beta >0$, $r>d/2-1$,  and let $X$ be as in Theorem 
\ref{th:ex} in the case  of  Robin boundary value. 

Let  $V_h$  be a finite-dimensional linear subspace of   $ H^1(D)$,  and let $X_h$ be a  $V_h$-valued random variable  that  satisfies
\begin{equation}
\int_ D \nabla X_h  \cdot \nabla \phi dx + \lambda \int_ D 
X_h \phi dx + \beta \int_{\partial D} X_h \phi d\sigma =  \int_D  ( \widehat Q_h \dot W)  \phi  dx  
\end{equation}
 for all $\phi\in V_h$, where $Q_h$ is the orthogonal projection onto $V_h$ in 
 $L^2(D)$. Then 
\begin{equation}\label{eq:upb}
\mathbb E \Vert X- X_h\Vert_{H^{-r}(D)}^2 \leq C \sup_{\Vert f\Vert_{L^2(D)}\leq 1} 
\Vert u^f -u^f_h\Vert^2_{H^{1}(D)},
\end{equation}
where  $u^f$ is the solution of 
 \begin{eqnarray}\label{eq:sol1}
\begin{split} 
  \left\{
  \begin{array}{r l l}
- \Delta u^f  + \lambda u^f   &=  f   & {\rm   \,  in\,  } D\\
   \partial_n u^f  |_{\partial D} + \beta u^f |_{\partial D} &=  0  & {\rm  \, on\,  } \partial D,
 \end{array}
 \right.
 \end{split}
 \end{eqnarray} and  $u^f_h \in V_h$ is the Ritz-Galerkin 
approximation  defined by 
\begin{equation}\label{eq:sol2}
\int_ D \nabla u^f _h  \cdot \nabla \phi dx + \lambda \int_ D 
u^f _h \phi dx + \beta \int_{\partial D} u^f _h \phi d\sigma =  \int_D \phi f dx  
\end{equation}
 for all $\phi\in V_h$.

\end{theorem}
\begin{proof}
By Lemma \ref{lemma:projection},  the measurable linear forms $ \widehat Q_h \dot W \in V_h$ 
 a.s. By regularity of the load $\widehat Q_h \dot W $,  the standard proof of  existence and 
uniqueness of Ritz-Galerkin approximations holds also for  $X_h$. Moreover, we can write 
$X_h=T_h (\widehat Q_h \dot W)$, where $T_h$ takes regular  loads  from  
$L^2(D)$ to corresponding Ritz-Galerkin approximations in $V_h$.

By Corollary \ref{cor:multi},  $T_h \widehat Q_h \dot W = \widehat {T_h Q_h } \dot W $, which almost surely coincides with $\widehat T_h \dot W$ since $\langle  Q_h  f, \phi\rangle = \langle f,\phi\rangle$ for all $f\in L^2(D)$ and  $\phi\in V_h$ 
and, by uniqueness of the solution, $T_h Q_h =T_h$ on $L^2(D)$.

Denote with $\{\phi_k\}$ an orthonormal basis in $H^r(D)$, and with $\Vert L \Vert_{HS: V_1\rightarrow V_2}$ the Hilbert-Schmidt norm of a linear mapping $L$ from  a separable Hilbert space $V_1$ into a separable Hilbert space $V_2$. Then by 
Corollary \ref{cor:multi} (ii),
\begin{equation}\label{eq:oneliner}
\begin{split}
\mathbb E  \Vert X- X_h\Vert^2_{H^{-r}(D)} &= 
\sum_{k=1}^\infty  \mathbb E \langle X- X_h , \phi_k\rangle_{H^{-r}(D),H^r(D)}^2 \\
 & = \sum_{k=1}^\infty  \Vert (T^*- T^*_h)  \phi_k\Vert_{L^2(D)}^2 
 \\
 &=    \Vert T^*- T^*_h \Vert_{HS: H^{r}(D) \rightarrow L^2(D)}^2\\
 &\leq \Vert T^* - T^*_h\Vert _{H^{-1}(D),L^2(D)}\Vert I \Vert_{HS:H^r(D)\rightarrow H^{-1}(D)},
\end{split}
\end{equation}
where the embedding  $I$ of $H^{r}(D)$ into $H^{-1}(D)$ is Hilbert-Schmidt  
by Maurin's theorem. Moreover,
$$
 \Vert T^* - T^*_h\Vert_{H^{-1}(D),L^2(D)}=  \Vert T- T_h \Vert_{L^2(D),H^{1}(D)}
= \sup_{\Vert f\Vert_{L^2(D)}\leq 1} \Vert u^f - u^f_h\Vert_{H^{1}(D)}.
$$
\qed\end{proof} 

\begin{corollary} \label{cor:lipp}
Let the assumptions of Theorem \ref{the:convergence} hold.
If  the finite-dimensional  subspaces $V_h$   fulfill the condition
$$\lim_{h\rightarrow 0} \min _{v\in V_h} \Vert u- v \Vert_{H^1(D)} =0$$
for all  $u\in H^1(D)$,   then  $X_h$ converges to  $X$   in   $L^2(\Omega,\Sigma,P; H^{-r}(D))$ as $h\rightarrow 0$. 
\end{corollary}
\begin{proof}
Let $u^f $ and $u^f_h$ satisfy \eqref{eq:sol1} and \eqref{eq:sol2}, respectively.
By Cea's lemma
\begin{eqnarray*}
 \sup_{\Vert f\Vert _{L^2(D)} \leq 1 } \Vert u^f  - u_h^f  \Vert_{ H^{1}(D)} 
  &\leq &   \sup_{\Vert f\Vert_{L^2(D)}\leq 1} \min_{ v\in V_h }\Vert   u^f  - v \Vert_{H^{1}(D)} ^2\\
&=&\sup_{\Vert f\Vert_{L^2(D)}\leq 1} \Vert (I-Q_h) u^f\Vert_{H^1(D)},
\end{eqnarray*}
where $Q_h$ is the orthogonal projection in $H^1(D) $ onto $V_h$.

By $H^1$-regularity of the solution $u^f$ of the  Robin BVP, the 
corresponding Neumann problem has boundary data $-\beta u^f|_{\partial D} 
\in H^\frac{1}{2}(\partial D)$. By Theorem 4 in \cite{MR1600081}, the unique 
solution  $u^f$ of this Neumann problem  belongs to $\in H^{1+s}(D)$  for some $0<s<\frac{1}{2}$.

The embedding of $H^{1+s}(D)\rightarrow  H^1(D) $ is compact. Under the mapping $A:=f\mapsto u^f$, the image of the 
unit ball  of $L^2(D)$ is relatively compact in $H^1(D)$.  For fixed  $\varepsilon>0$,  all  open balls  $B(\phi ,\varepsilon)$ of $H^1(D)$ whose center points  $\phi$ belong to the image of the unit ball under the mapping $A$,  form a  cover of the image  and  we may choose a finite subcover 
consisting of sets $B(\phi_k, \varepsilon)$, $k=1,\dots,K$.    For any $f$ satisfying 
$\Vert f\Vert_{L^2(D)}\leq 1$, we have  
\begin{eqnarray*}
\Vert (I-Q_h ) u^f \Vert _{H^1(D)} &\leq &  \Vert  u^f -\phi_k \Vert_{H^1(D)}   +
\sum_{k=1}^K \Vert (I-Q_h ) \phi_k  \Vert_{H^1(D)}\\
&\leq & \varepsilon + \sum_{k=1}^K
  \Vert (I-Q_h ) \phi_k  \Vert_{H^1(D)} 
\end{eqnarray*}
for some $k=1,\dots,K$. 
\qed\end{proof}

\begin{remark}
Replacing $H^1(D)$ with $H^1_0(D)$  in Theorem 
\ref{the:convergence} gives the same results for the 
homogeneous Dirichlet problem (see \cite{MR1600081} for 
the  required  regularity). 
\end{remark}

Many practical applications of boundary value problems involve finite element methods.  Below, we demonstrate how to estimate the  speed of  convergence of  finite   element  approximations. 
 For simplicity, 
we consider only two dimensional polygonal domains with triangular elements.   For details of 
finite element methods we refer to \cite{brenner}. 

The following corollary shows how  the speed of convergence  in finite element approximations depends on the regularity of the solution.   We emphasize that  the domain is allowed to be  nonconvex. 
\begin{corollary}
Let $X$ be as in Theorem \ref{th:ex} for $d=2$.
Let ${\mathcal T} _h$ be a  triangulation of a bounded polygonal domain  $D\subset \mathbb R^2$ into  elements of diameter at most  $h>0$ and let $V_h\subset H^1(D)$    be the corresponding finite element space consisting of  continuous functions.  Let the interpolant 
 $I_h$ satisfy
$
\Vert u -I_h u\Vert_{H^1(D)} \leq C h \Vert u\Vert_{H^2(D)}.  
$  
 Then there exists  $0<s<\frac{1}{2}$  such that the finite element approximations $X_h$ of  the solution 
$X$ satisfy
$$
\mathbb E \Vert X- X_h\Vert_{H^{-r}(D)} ^2 \leq  C h^{2s}.
$$
 \end{corollary}
 \begin{proof}
 The solution $u^f$  to the homogeneous elliptic boundary value problem 
 \eqref{eq:bvp}   belongs  to $H^{1+s}(D)$ for some $s>0$ \cite{MR1600081} and    satisfies  $$
\Vert u^f \Vert_{H^{1+s}(D)}\leq c \Vert f\Vert_{L^2(D)},
$$
for all $f\in L^2(D)$.  The assumptions imply then that (e.g. \cite{brenner},  Chapter 12.3) 
 $$
 \Vert u^f - u_h^f \Vert_{H^1(D)}\leq  C h^{s} \Vert u^f \Vert_{H^{1+s}(D)},
 $$
where 
 $$
 \Vert u^f \Vert _{H^{1+s}(D)} \leq C \Vert f \Vert_{L^2(D)}.
 $$
 The last claim follows then from \eqref{eq:upb}.
 \qed\end{proof}

Next, we  improve the  previous convergence results  by replacing  $H^{-r}(D)$ spaces with  
 the space of continuous functions $C(\overline D)$ equipped with the 
usual supremum norm.  Instead of $L^2$-convergence, we study  weak convergence of probability distributions.  We apply the notations 
appearing in Theorem \ref{the:convergence}.
\begin{theorem}
Let $D\subset \mathbb R^2$ be a bounded Lipschitz domain. 
 Let $X$ be as in Theorem \ref{th:ex} and let $X_h$ be its finite element approximations.  Assume that the 
the solution $u^f$  of \eqref{eq:bvp} and its finite element approximations $u^f_h$   are such that 
$$
\lim_{h\rightarrow 0 }\sup_{\Vert f\Vert_{L^2(D)\leq 1}}  |u_h^{f}(x)-u^f(x)|=0
$$
for  any $x\in \overline D$ and there exists $\alpha,C>0$ such that
\begin{equation}\label{eq:yksi}
\sup_{\Vert f\Vert_{L^2(D)\leq 1}}|u_h^f (x)- u_h^f (y) |\leq  C |x-y|^\alpha
\end{equation}
for all $0<h<1$ and $x,y \in \overline D$.

Then the probability distribution of    $X_h$ converges to  the probability distribution of 
$X$ weakly in the sense of measures on $C(\overline{D})$.    
\end{theorem}

\begin{proof}
It is clear from \eqref{eq:yksi} that $X_h$ has continuous realizations.

We show that  the finite-dimensional probability  distributions of $X_h$ converge to the finite-dimensional probability distributions of 
$X$,  and that the family of  probability distributions of $X_h$ on $C(\overline{D})$ is tight.

The convergence of finite-dimensional distributions of zero mean Gaussian random variables $ (X_h  ({x_1}), \dots , 
X_h ( {x_k}))$ depends only 
on the convergence of the covariances
$$
\mathbb E X_h (x)  X_ {h} ({y}). 
$$  
Note that  $X_{h}(x) =\langle X_h ,\delta_x\rangle$. 
Moreover, 
\begin{equation}\label{eq:pres}
X_h = \widehat T_h \dot W ,
\end{equation}
 where $T_h f$ solves the variational boundary value problem for $f\in L^2(D)$. 
Hence 
\begin{eqnarray*}
\mathbb E  X_ h ({x})  X_ {h}({y}) &=&  \mathbb E  \langle X_h ,\delta _x \rangle \langle X_h , \delta_y\rangle 
=  \mathbb  E \langle \dot W, T_h^* \delta _x \rangle \langle \dot W,  T_h^*\delta_y\rangle  
=  \langle T_h^* \delta_x, T_h ^*\delta _y\rangle. 
\end{eqnarray*}
We show that  $T_h^* \delta_{x}$ converge to $T^* \delta_{x}$ in  $L^2(D)$. 
Indeed, $$
\Vert T_h^* \delta_x - T ^* \delta_ x\Vert _{L^2(D)} = 
\sup_{\Vert f\Vert_{L^2(D)}\leq 1}   | u_h ^f (x) - u^f (x)|  
$$
which converges to zero by the assumptions. Therefore, the finite dimensional 
distributions converge.

Next, we verify  tightness by Kolmogorov's theorem (see \cite{schilling}). 
By linearity
\begin{equation}\label{eq:delta}
X_h (x) - X_h  (y)= \langle X_h , \delta_x- \delta_y\rangle.
\end{equation}
We insert \eqref{eq:pres}  and   \eqref{eq:delta} into  
\begin{eqnarray*}
\mathbb E |X_h(x) -X_h(y) |^2 &=& \mathbb E  \langle X_h  , \delta_x-\delta_y\rangle ^2 = \mathbb E  \langle \dot W , T_h^* (\delta_x-\delta_y)\rangle ^2\\
& =&   \Vert T^* _h ( \delta_x-\delta_y)\Vert ^2 _{L^2(D)}.
\end{eqnarray*}
 By the assumptions,  
\begin{eqnarray*} 
\mathbb E  | X_h(x) -X_h(y) |^2   = \sup_{\Vert f\Vert_{L^2(D)}\leq 1 }  |u_h^f (x)- u_h^f (y) |^2 \leq   C |x-y|^{2\alpha}.
\end{eqnarray*}
Similarly
$
\mathbb  E  |X_h(x)| ^2  \leq   C.
$
The identification of the limit with the probability distribution of $X$ is carried out with characteristic functions.
\qed\end{proof}

\section{Conclusions}

The presented methodology is an effective tool tailored for Gaussian problems and does not directly generalize to nonlinear elliptic problems. However, the principle idea of replacing the normal derivative at the boundary with a measurable mapping may  carry over to more general problems.

\begin{acknowledgements}
This work has been funded 
by Academy of Finland (application number 250215, Finnish Programme for Centre of Excellence in Research 2012-2017) and European Research Council (ERC advanced grant 267700 - Inverse problems). 
\end{acknowledgements}

\bibliographystyle{spmpsci}      

\bibliography{mybib}{}

\begin{thebibliography}{10}
\providecommand{\url}[1]{{#1}}
\providecommand{\urlprefix}{URL }
\expandafter\ifx\csname urlstyle\endcsname\relax
  \providecommand{\doi}[1]{DOI~\discretionary{}{}{}#1}\else
  \providecommand{\doi}{DOI~\discretionary{}{}{}\begingroup
  \urlstyle{rm}\Url}\fi

\bibitem{allen}
Allen, E.J., Novosel, S.J., Zhang, Z.: Finite element and difference
  approximation of some linear stochastic partial differential equations.
\newblock Stochastics Stochastics Rep. \textbf{64}(1-2), 117--142 (1998)

\bibitem{babuska2}
Babu{\v{s}}ka, I., Nistor, V.: Boundary value problems in spaces of
  distributions on smooth and polygonal domains.
\newblock J. Comput. Appl. Math. \textbf{218}(1), 137--148 (2008)

\bibitem{babuska}
Babu{\v{s}}ka, I., Tempone, R., Zouraris, G.E.: Galerkin finite element
  approximations of stochastic elliptic partial differential equations.
\newblock SIAM J. Numer. Anal. \textbf{42}(2), 800--825 (2004)

\bibitem{benfatto}
Benfatto, G., Gallavotti, G., Nicol{\`o}, F.: Elliptic equations and {G}aussian
  processes.
\newblock J. Funct. Anal. \textbf{36}(3), 343--400 (1980)

\bibitem{ben}
Benth, F.E., Gjerde, J.: Convergence rates for finite element approximations of
  stochastic partial differential equations.
\newblock Stochastics Stochastics Rep. \textbf{63}(3-4), 313--326 (1998)

\bibitem{bog}
Bogachev, V.I.: Gaussian measures, \emph{Mathematical Surveys and Monographs},
  vol.~62.
\newblock American Mathematical Society, Providence, RI (1998)

\bibitem{bogd}
Bogachev, V.I.: Differentiable measures and the {M}alliavin calculus,
  \emph{Mathematical Surveys and Monographs}, vol. 164.
\newblock American Mathematical Society, Providence, RI (2010)

\bibitem{brenner}
Brenner, S.C., Scott, L.R.: The mathematical theory of finite element methods,
  \emph{Texts in Applied Mathematics}, vol.~15.
\newblock Springer-Verlag, New York (1994)

\bibitem{pardoux}
Buckdahn, R., Pardoux, {\'E}.: Monotonicity methods for white noise driven
  quasi-linear {SPDE}s.
\newblock In: Diffusion processes and related problems in analysis, {V}ol.\ {I}
  ({E}vanston, {IL}, 1989), \emph{Progr. Probab.}, vol.~22, pp. 219--233.
  Birkh\"auser Boston, Boston, MA (1990)

\bibitem{cao}
Cao, Y., Yang, H., Yin, L.: Finite element methods for semilinear elliptic
  stochastic partial differential equations.
\newblock Numer. Math. \textbf{106}(2), 181--198 (2007)

\bibitem{cao2}
Cao, Y., Zhang, R., Zhang, K.: Finite element and discontinuous {G}alerkin
  method for stochastic {H}elmholtz equation in two- and three-dimensions.
\newblock J. Comput. Math. \textbf{26}(5), 702--715 (2008)

\bibitem{caoscat}
Cao, Y., Zhang, R., Zhang, K.: Finite element method and discontinuous
  {G}alerkin method for stochastic scattering problem of {H}elmholtz type in
  {$\Bbb R^d$} {$(d=2,3)$}.
\newblock Potential Anal. \textbf{28}(4), 301--319 (2008)

\bibitem{MR0162126}
Donoghue Jr., W.F.: On a theorem of {K}. {M}aurin.
\newblock Studia Math. \textbf{24}, 1--5 (1964)

\bibitem{qi}
Du, Q., Zhang, T.: Numerical approximation of some linear stochastic partial
  differential equations driven by special additive noises.
\newblock SIAM J. Numer. Anal. \textbf{40}(4), 1421--1445 (electronic) (2002)

\bibitem{franklin}
Franklin, S.R., Seshaiyer, P., Smith, P.W.: A three-field finite element method
  for elliptic partial differential equations driven by stochastic loads.
\newblock Stoch. Anal. Appl. \textbf{23}(4), 757--783 (2005)

\bibitem{MR775683}
Grisvard, P.: Elliptic problems in nonsmooth domains, \emph{Monographs and
  Studies in Mathematics}, vol.~24.
\newblock Pitman (Advanced Publishing Program), Boston, MA (1985)

\bibitem{kahane}
Kahane, J.P.: Some random series of functions, \emph{Cambridge Studies in
  Advanced Mathematics}, vol.~5, second edn.
\newblock Cambridge University Press, Cambridge (1985)

\bibitem{scales}
Krein, S.G., Petunin, Y.I.: Scales of banach spaces.
\newblock Russian Mathematical Surveys \textbf{21}(2), 85 (1966).
\newblock \urlprefix\url{http://stacks.iop.org/0036-0279/21/i=2/a=R03}

\bibitem{mclean}
McLean, W.: Strongly elliptic systems and boundary integral equations.
\newblock Cambridge University Press, Cambridge (2000)

\bibitem{MR2812574}
Nittka, R.: Regularity of solutions of linear second order elliptic and
  parabolic boundary value problems on {L}ipschitz domains.
\newblock J. Differential Equations \textbf{251}(4-5), 860--880 (2011)

\bibitem{MR3209311}
Roininen, L., Huttunen, J.M.J., Lasanen, S.: Whittle-{M}at\'ern priors for
  {B}ayesian statistical inversion with applications in electrical impedance
  tomography.
\newblock Inverse Probl. Imaging \textbf{8}(2), 561--586 (2014)

\bibitem{rozanov}
Rozanov, Y.A.: Stochastic {S}obolev spaces and their boundary trace.
\newblock Theory Probab. Appl. \textbf{40}(1), 104--115 (1996)

\bibitem{MR1600081}
Savar{\'e}, G.: Regularity results for elliptic equations in {L}ipschitz
  domains.
\newblock J. Funct. Anal. \textbf{152}(1), 176--201 (1998)

\bibitem{schilling}
Schilling, R.L.: Sobolev embedding for stochastic processes.
\newblock Expo. Math. \textbf{18}(3), 239--242 (2000)

\bibitem{swap}
Schwab, C., Todor, R.A.: Sparse finite elements for elliptic problems with
  stochastic loading.
\newblock Numer. Math. \textbf{95}(4), 707--734 (2003)

\bibitem{walsh2}
Walsh, J.B.: A stochastic model of neural response.
\newblock Adv. in Appl. Probab. \textbf{13}(2), 231--281 (1981)

\bibitem{walsh}
Walsh, J.B.: An introduction to stochastic partial differential equations,
  \emph{Lecture Notes in Math.}, vol. 1180.
\newblock Springer, Berlin (1986)

\bibitem{zhang}
Zhang, K., Zhang, R., Yin, Y., Yu, S.: Domain decomposition methods for linear
  and semilinear elliptic stochastic partial differential equations.
\newblock Appl. Math. Comput. \textbf{195}(2), 630--640 (2008)

\end{thebibliography}

\end{document}